\documentclass[11pt]{amsart}
\usepackage{amsmath,amsfonts,latexsym,amssymb,amsthm, verbatim}
\usepackage{url}
 \usepackage[OT2, T1]{fontenc}

\usepackage{resizegather}

\usepackage{amscd}

\usepackage{caption,booktabs}
\usepackage[]{amsrefs}
\usepackage{setspace}

\usepackage{outlines}
\usepackage[flushleft]{threeparttable}

\usepackage{cleveref}

\newtheorem{thm}{Theorem}
\newtheorem{prop}[thm]{Proposition}
\newtheorem{lem}[thm]{Lemma}

\theoremstyle{definition}

\newcommand{\BQ}{\mathbb{Q}}

\newcommand{\BZ}{\mathbb{Z}}

\begin{document}

\title[]{Growth of Torsion of Elliptic Curves with Odd-Order Torsion over Quadratic Cyclotomic Fields}

\author{Burton Newman}
\address{Department of Computer Science\\ University of Southern California\\ Los Angeles, CA 90089\\ USA}
\email{bnewman@usc.edu}

\thanks{The author was partially supported by the NSF grant DMS-1302886.}

\keywords{Elliptic curves, torsion subgroups, modular curves}
\subjclass[2010]{Primary: 11G05, Secondary: 14G35}

\begin{abstract}

Let $K = \mathbb{Q}(\sqrt{-3})$ or $\mathbb{Q}(\sqrt{-1})$ and let $C_n$ denote the cyclic group of order $n$. We study how the torsion part of an elliptic curve over $K$ grows in a quadratic extension of $K$. In the case $E(K)[2] \approx C_1$ we investigate how a given torsion structure can grow in a quadratic extension and the maximum number of extensions in which it grows.  We also study the torsion structures which occur as the quadratic twist of a given torsion structure.  In order to achieve this we  examine $n$-isogenies 
defined over $K$ for $n=15,20,21,24,27,30,35$.

\end{abstract}
\maketitle

\section{Introduction}

\noindent 

Let $K$ be a number field and $E/K$ an elliptic curve. An $n$-cycle of $E/K$ is a cyclic subgroup of $E(\overline{K})$ of order $n$ which is invariant under the action of Gal($\overline{K}/K$).  
An $n$-cycle $C$ of $E/K$ gives rise to a curve $E'/K$ and an isogeny $E \rightarrow E'$ defined over $K$  with kernel cyclic of order $n$,
and every such isogeny 
arises this way \cite{Si}*{Rmk 4.13.2}.  If $E/K$ has an isogeny of this form we say $E/K$ has an $n$-isogeny. 
If the points of an $n$-cycle are rational over an extension $L/K$, we will say the corresponding $n$-isogeny is \textit{pointwise rational} over $L$.
Let $C_n$ denote the cyclic group of order $n$.

In this paper, we  classify $n$-isogenies defined over $K= \mathbb{Q}(\sqrt{-3})$ or $\mathbb{Q}(\sqrt{-1})$  that are pointwise rational over quadratic extensions of $K$ for $n =15,20,21,$ $24,27,30,35$  (with one exception)(\cref{isogenies}).  In the case $K=\mathbb{Q}(\sqrt{-3})$, $E(K)[2]=C_1$ we determine (i) a classification of the torsion structures  which occur as the quadratic twists of a given torsion structure, (ii) a classification of the torsion structures which occur as the growths of a given torsion structure  and (iii) tight bounds on the number of quadratic extensions in which a given torsion structure can grow (\cref{oddtwist}).
In the case $K= \mathbb{Q}(\sqrt{-1})$ we did not complete the classification because we could not disprove the existence of a 21-isogeny over $K$.  This  was accomplished in \cite{Ejder}*{Prop. 2}.
For a history of related classification problems please see \cite{Newman}.

There is an affine curve $Y_0(n)$ whose $K$-rational points classify isomorphism classes of pairs $(E,C)$ where $E/K$ is an elliptic curve and $C$  is an $n$-cycle.  Two pairs $(E,C)$,$(E',C')$ are equivalent if and only if there is an isomorphism $f:E \rightarrow E'$ such that $f(C)= C'$.  By adding a finite number of points (called $\emph{cusps}$) to $Y_0(n)$ we obtain the projective curve $X_0(n)$.  The curve $X_0(n)$ has a model over $\mathbb{Q}$ and hence we have tools to study $X_0(n)(K)$.

Let $E/K$ be an elliptic curve and let $L/K$ be a quadratic extension. We summarize our strategy as follows: When $E(K)[2] = C_1$, we have $E(L)[2] = C_1$ (\cref{GT}).  Hence in order to complete tasks (i), (ii) and (iii), it suffices to complete (i) by \cref{oddpartsadd}. 
 But by \cref{twistback}, if $E(K)_{tor} \not = C_1$ and  $E^d(K)_{tor} \not= C_1$ then one can often show $E$ has an $n$-isogeny pointwise rational over $K(\sqrt{d})$ for some large value of $n$, and these are rare over $K$ (\cref{isogenies}).  The classification of $n$-isogenies leads to the Diophantine problem of determining $X_0(n)(K)$.

In Section 2 we describe some results necessary to understand the rest of the paper. In Section 3 we study the $K$-rational points on $X_0(n)$ for $K = \mathbb{Q}(\sqrt{-1})$, $\mathbb{Q}(\sqrt{-3})$ and $n = 15,20,21, 24,27,30,35$. In Section 4 we use the  classification of $n$-isogenies to study growth of torsion.

Computation played an important role in our work.  We used Magma to compute the rank and torsion of elliptic curves over number fields.
We also used Magma to find automorphism groups of  curves and compute quotient curves under the action of certain groups.  The classification of $n$-isogenies relied upon the Small Modular Curves package in Magma.

\section{Background}
\label{Background}

We require the following classification theorem.

\begin{thm}({Najman \cite{NajCyc}})
\label{Najman}
Let K be a cyclotomic quadratic field and E an elliptic curve over K.  
\begin{itemize}
\item If  $K = \mathbb{Q}(i)$ then $E(K)_{tor}$ is either one of the groups from Mazur's theorem \cite{Mazur}*{Thm 2} or $C_4 \oplus C_{4}$.
\item  If  $K = \mathbb{Q}(\sqrt{-3})$ then $E(K)_{tor}$ is either one of the groups from Mazur's theorem, $C_3 \oplus C_{3}$ or $C_3 \oplus C_{6}$.
\end{itemize}
\end{thm}

\begin{prop}\cite{Tornero}*{Cor. 4}
\label{oddpartsadd} If n is an odd positive integer we have 
$$E(K(\sqrt{d}))[n] \approx E(K)[n] \oplus E^d(K)[n]$$
\end{prop}

\begin{prop}
\label{twistback}
Let K be a number field and E/K an elliptic curve. 
Let d $\in$ K be a nonsquare and let L = K$(\sqrt{d})$.  If H is a subgroup of $E^d(K)_{tor}$ of odd order, then there is a Gal($\overline{K}/K$)-invariant subgroup J of E(L)$_{tor}$ such that J $\approx H \oplus E(K)_{tor}$.

\end{prop}

\begin{proof}
We may assume E is in  Weierstrass form.  We have an isomorphism:  $$ T: E^d \rightarrow E$$ $$ (x,y) \mapsto (x,\sqrt{d}y)$$  so $T(H) \approx H$ and $T(H)$ is Gal($\overline{K}/K$)-invariant
 since the points of $H$ are rational over $K$ and $H$ is a subgroup (so closed under inverses).  Since $H$ has odd order, it has no points of order $2$ so $T(H) \cap E(K)_{tor} = \{0\}$.  Hence $J := T(H) + E(K)_{tor} \approx H \oplus E(K)_{tor}$ is a subgroup of $E(L)$. As $J$ is the sum of Gal$(\overline{K}/K)$-invariant subgroups, it is invariant as well.
\end{proof}

\begin{prop}
\label{full}

Let K$=\mathbb{Q}(\sqrt{D})$ ($D = -1, -3$), E/K an elliptic curve and L a quadratic extension of K. Then the only odd prime power n such that $C_n \oplus C_n  \subseteq E(L) $ is n=3.

\end{prop}

\begin{proof}

Let $\phi$ denote Euler's totient function.  If $n=p^t$ where $p$ is a prime and $C_n \oplus C_n  \subseteq E(L) $ then by \cite{Si}*{Cor 8.1.1}
, $L$ contains an $n$th root of unity $\mu_n$.  Hence $ (p-1)p^{t-1} =  \phi(n) \leq [L:\mathbb{Q}] =4$ so $n=2,3,4,5$ or $8$. Note there is either a 3rd or 4th root of unity in $K$, 
so if $\mu_5$ is in $L$, then there is a 15th or 20th root of unity in $L$. 
But $\phi(15)=8 > 4$ and $\phi(20)=8 >4$, a contradiction.  
On the other hand, there is an elliptic curve E/$\mathbb{Q}$ (namely $[0,-1,1, 217, -282]$) which has full 3-torsion over K$=\mathbb{Q}(\sqrt{-3})$ and hence provides examples in each case with L$=\mathbb{Q}(\sqrt{-1},\sqrt{-3})$.
\end{proof}

The following theorem lists various restrictions on growth in quadratic extensions.

\begin{thm}
\label{GT} 
Let K be a number field, E/K an elliptic curve, L a quadratic extension of K and p an odd prime.

\begin{enumerate}
\item If E(K)[2] $ \approx C_1$ then E(L)[2] $ \approx C_1$ . \label{GT 1}

\item  If d $\in$ K, d $\ne 0$, then $E^d(K)[2] \approx E(K)[2]$.

\item If E(K)[p] $ \approx C_1$ and E(L)[p] $\approx C_p \oplus C_p$ then K contains a primitive pth root of unity.

\item If E(K)[p] $\approx C_p$ and E(L)$[p^{\infty}]\ne$ E(K)$[p^{\infty}]$ then E(L)[p] $\approx C_p \oplus C_p$.

\item If E(K)[p] $\approx C_p$ and  E(L)[p] $\approx C_p \oplus C_p$ then K does not contain a primitive pth root of unity.
\label{p to pp}

\item If E(K)[p] $\approx C_p \oplus C_p$ then E(L)$[p^{\infty}]= $ E(K)$[p^{\infty}]$.

\end{enumerate}\end{thm}

\begin{proof} 

Parts (1) and (2) are easily verified.

3) Suppose $E(K)[p]$ is trivial.
By \cref{oddpartsadd}, it follows that $E^d(K)[p] \approx C_p \oplus C_p$ so by \cite{Si}*{Cor 8.1.1} we conclude $K$ contains a primitive $p$th root of unity.

4) Let $m$ the largest positive integer such that there is an element of order $p^m$ in $E(L)_{tor}$. We have   $E(L)[p^{m}]= E(K)[p^{m}] \oplus E^d(K)[p^{m}]$ by \cref{oddpartsadd}. If 
$E(L)[p^{m}]\ne E(K)[p^{m}]$ then $E^d(K)[p^{m}] \not \approx C_1$ so $E^d(K)[p]$ $\not \approx C_1$.  Hence $E^d(K)[p] \approx C_p$ or $C_p \oplus C_p$ by 
\cite{Si}*{Cor. 6.4}.  In the latter case this would yield $E(L)[p] \approx C_p \oplus C_p \oplus C_p$ which contradicts 
\cite{Si}*{Cor. 6.4}, so $E(L)[p] \approx C_p \oplus C_p $.

5) Let $\mu_p$ be a primtive $p$th root of unity. Suppose $E(K) \approx C_p$, and $E(L)[p] \approx C_p \oplus C_p$.  Let $\sigma  \in $ Gal($L/K$) be nontrivial.  We can choose a basis for $E(L)[p]$ such that the induced Galois representation satisfies
$$ \rho: \mbox{Gal(L/K)} \rightarrow \mbox{Gl}_2(\BZ / p \BZ)$$
$$\sigma \mapsto 
\begin{bmatrix}
1 & \alpha\\
0 & \chi 
\end{bmatrix}$$

for some $\chi \in (\BZ / p \BZ)^{*}$, $\alpha \in \BZ / p \BZ$. If $\mu_p \in K$ then $$\mu_p = \sigma(\mu_p) = (\mu_p)^{det(\rho(\sigma))}=(\mu_p)^{\chi}$$ 
so $\chi = 1$ mod $p$.  As $\sigma^2=1$, $(\rho(\sigma))^2=1$ so $2\alpha = 0$. As $p$ is odd, we  conclude $\alpha=0$, so $\rho(\sigma)$ is the identity.  This means $\sigma$ acts trivially on the p-torsion, so $E(K)[p] \approx C_p \oplus C_p$, contradicting our hypothesis.

6) Suppose $E(K)[p] \approx C_p \oplus C_p$. By \cref{oddpartsadd}, if $E(L)[p^{\infty}] \ne  E(K)[p^{\infty}]$ then $E^d(K)[p^{\infty}] \not \approx C_1$ so $E^d(K)[p]$ $\not \approx C_1$.  Hence $C_p \oplus C_p \oplus C_p \subseteq E(L)[p]$, contradicting \cite{Si}*{Cor. 6.4}.
\end{proof}

Note that there are growths which occur over $\BQ$ but not over some  quadratic extension:  $C_3$ to $C_3 \oplus C_3$ occurs over $\BQ$ but by \cref{GT} Part 5 not over $\BQ(\sqrt{-3})$.  On the other hand, there are growths which occur over a quadratic field but not over $\BQ$: $C_1$ to $C_3 \oplus C_3$ cannot over $\BQ$ because if it did, $\BQ$ would contain a primitive 3rd root of unity by \cref{GT} Part 3.  On the other hand, $C_1$ to $C_3 \oplus C_3$ occurs over $\BQ(\sqrt{-3})$ 

Also,
$C_1$ to $C_{15}$ cannot occur over $\BQ$ by \cref{oddpartsadd} since $X_1(15)$ has no noncuspidal $\BQ$-rational points.  On the other hand, this growth does occur over $K=\BQ(\sqrt{5})$:  By \cite{NajCubic}*{Thm 2} there is an elliptic curve E/K with $E(K)_{tor} \approx C_{15}$.  Choose d $\in K$, $d \not = 0$, such that $E^d(K)_{tor} = C_1$.  Then $E(K(\sqrt{d}))_{tor} \approx C_{15}$ by \cref{oddpartsadd} and \cref{GT} Part 1.

\section{$K$-Rational Points on $X_0(n)$}
\label{ratpoints}

To study torsion over quadratic fields in the case $j=0, 1728$ we use the technique from \cite{Lem}.

\begin{lem}
\label{count}
Let p be a prime and $E/F_p$ an elliptic curve with model $y^2=x^3+Ax+B$.

\begin{enumerate}

\item If A=0 (i.e. j(E)= 0) and p $\equiv$ 2 mod 3  then $|E(F_p)| = p+1$ and $|E(F_{p^2})| = (p+1)^2$.

\item If B=0 (i.e. j(E)=1728) and p $\equiv$ 3 mod 4  then $|E(F_p)| = p+1$ and $|E(F_{p^2})| = (p+1)^2$.

\end{enumerate}

\end{lem}

\begin{proof}

If $A=0$  and $p \equiv 2$ mod $3$  then $|E(F_p)| = p+1$ by \cite{Washington}*{Prop 4.33}.  If B=0 and $p \equiv 3$ mod $4$  then $|E(F_p)| = p+1$ by \cite{Washington}*{Thm 4.23}.  Now in either case above, $|E(F_{p^2})| = p^2+ 1-(\alpha^2+\beta^2)$ by \cite{Washington}*{Thm 4.12}, where $\alpha$ and $\beta$ are roots of $x^2+p$.  Hence  
\begin{align*}
|E(F_{p^2})| & = p^2+ 1-(\alpha^2+\beta^2)\\
& =  p^2+ 1 - (-p-p)\\
& = p^2+2p+1\\
& = (p+1)^2\\
\end{align*}
\end{proof}

\begin{thm}
\label{j=0}
\label{j=1728}
Let $K$ be a quadratic field and $E/K$ an elliptic curve.  If $j(E)=0$ and $p>3$ is a prime then 
$E(K)_{tor}$ has no element of order $p$.  If $j(E)=1728$ and $p>2$ is a prime then 
$E(K)_{tor}$ has no element of order $p$.

\end{thm}

\begin{proof}
Suppose $j(E)=0$.  Twisting by a square in $\mathcal{O}_K$ if necessary, we may assume $E$ has a model of the form  $y^2=x^3+AX+B$ with $A,B \in \mathcal{O}_K$.  Note that since $\mathcal{O}_K$ is a Dedekind domain, the principal ideal (disc($E$)) has only a finite number of prime ideal divisors, and hence disc($E$) lies in only a finite number of prime ideals of $\mathcal{O}_K$.  Let  $q > 3$ be a prime in $\mathbb{Z}$.  Since $q$ $\not = 3$, by the Chinese remainder theorem there exists an integer $n$ satisfying:$$n+1 \equiv 2 \mbox{ mod } q$$ $$ n \equiv 2 \mbox{ mod } 3$$
Furthermore, $n+3qk$ satisfies the congruences above for every integer $k$, and $(n,3q)=1$ by the congruences above.  Hence by Dirichlet's theorem on arithmetic progressions, there are infinitely many primes in this arithmetic progression.  In particular, there is a prime $p$ satisfying the congruences above such that $E$ has good reduction modulo a prime ideal $\beta$ above $p$.  As $[K:\mathbb{Q}]=2$, we have $\mathcal{O}_K/\beta \approx F_p$ or $F_{p^2}$.
By the comments following \cite{Si}*{Prop. 3.1} 
we have an injection of the group $E(K)[\overline{p}]$ into 
$E(F_p)$ or $E(F_p^2)$.  But by \cref{count} we have: $$|E(F_p)| = p+1 \equiv 2 \not \equiv 0 \mbox{ mod } q$$ $$|E(F_{p^2})| = (p+1)^2 \equiv 4 \not \equiv 0 \mbox{ mod } q$$ as  q $\not = 2$.  Hence in either case (noting $p \not = q$), we conclude there is no point of order $q$ in $E(K)_{tor}$.

Now suppose $j=1728$.  If $q$ is an odd prime, then one can argue just as in the $j=0$ case that there is no point of order $q$.  
\end{proof}

\subsection{}

Magma describes the $n$-cycle $C$ corresponding to an $n$-isogeny 
by providing a polynomial $f_C$ whose roots are precisely the $x$-coordinates of the points in $C$.  Given an $n$-isogeny with $n$-cycle $C$, let $K_C$ denote
the field of definition of $C$ (that is, the field obtained by adjoining to $K$ all the coordinates of the points of $C$) and for a polynomial $f$, let $K(f)$ denote the splitting field of $f$ over $K$.  If an $n$-isogeny with $n$-cycle $C$ is pointwise rational over a field $L$ then $f_C$ should split completely over $L$.  In particular if $L$ is a quadratic extension of $K$, then $f_C$ must have irreducible factors of degree at most $2$ over $K$.

We will now argue that no elliptic curve over $K=\mathbb{Q}(\sqrt{-3})$ has a $21$-isogeny pointwise rational over a quadratic extension of $K$.  Magma tells us the modular curve $X_0(21)$ has rank $0$, torsion $C_2 \oplus C_8$ over $K$ and $4$ cusps over $K$.  The $12$ non-cuspidal points correspond to isomorphism classes $\overline{(E,C)}$ and using Magma we found representatives of each class (see \Cref{x0(21)}).  As one can see from the table, for each representative $(E,C)$ with $j \not = 0$, $f_C$ has an irreducible factor of degree at least $3$ and hence 
$[K_C:K] \geq [K(f_C):K] \geq 3$.
In particular, there is no quadratic extension $L/K$ such that   all the points of $C$ are $L$-rational.
Now since in each of the cases just mentioned, $j \not = 0,1728$, by \cite{Si}*{p. 45} 
the isomorphism class of $(E,C)$ just consists of $(E^d,C^d)$ for nonzero $d \in K$, where $C^d$ denotes the image of $C$ under quadratic twist by $d$. If $(x,y) \in C$ then $(dx, d^{3/2}y) \in C^d$.  As $d \in K$, $K(f_C) = K(f_{C^d})$.  Hence the $8$ isomorphism classes with $j \not = 0$ in \Cref{x0(21)} do not contain an example of an elliptic curve $E/K$ with a $21$-isogeny pointwise rational over a quadratic extension of $K$.  In the $j=0,1728$ case, Magma is not yet able to describe the isomorphism class, so we instead argue as follows:  If there is an elliptic curve $E/K$ with a $21$-isogeny pointwise rational over a quadratic extension $L$ of $K$, then by \cref{oddpartsadd} there is an elliptic curve $E/K$ with a point of order 7 over $K$. But this is impossible by \cref{j=0}.

{
\renewcommand{\arraystretch}{1.3}

\begin{table}[ht]
 \caption{($K=\mathbb{Q}(\sqrt{-3})$) Representatives (E,C) of isomorphism classes corresponding to non-cuspidal K-rational points on a model of $X_0(21)$ } 
 
\begin{threeparttable}
 
 \centering 

    \begin{tabular}{ | l | l | l | l |}
    
    \hline
\label{x0(21)}    
    Point & j(E) &  E &  $f_C$ \\ \hline
    
    $(-1/4, 1/8)$ & $3375/2$ & $[20/441, -16/27783 ]$ & $(1,3,3,3)$ \\ \hline
    
   $(2, -1)$  & $-189613868625/128$ & $[-1915/36, -48383/324]$  & $(1,3,6)$ \\ \hline
   
    $(-1, 2)$ &  $-1159088625/2097152$   &   $[-505/192, -23053/6912 ]$ &  $(1,3,6)$  \\ \hline
    
     $(5,13)$  &  $-140625/8$  & $[-1600/147, -134144/9261]$  & $(1,3,3,3)$ \\ \hline

    $(\frac{(\alpha+1)}{2}, \alpha-1)$ &  $-12288000$  &  $[\frac{(40\alpha+10)}{49}, \frac{(-2530\alpha-6831)}{12348}]$ & $(1,3,6)$ \\ \hline
    
     $(\frac{(-\alpha+1)}{2}, -\alpha-1)$ &  $-12288000$  &  $[\frac{(-40\alpha+10)}{49}, \frac{(2530\alpha-6831)}{12348}]$  & $(1,3,6)$ \\ \hline
    
    $(\frac{(\alpha+1)}{2}, \frac{(-3\alpha+1)}{2})$  &  $54000$   & $[\frac{-135\alpha-585}{98}, \frac{-660\alpha-1782}{343} ]$ & $(1,3,6)$  \\ \hline
    $(\frac{(-\alpha+1)}{2}, \frac{(3\alpha+1)}{2})$  &  $54000$   & $[\frac{135\alpha-585}{98}, \frac{660\alpha-1782}{343} ]$ & $(1,3,6)$  \\ \hline

    $(\frac{-3\alpha-5}{2}, 8)$  &  0 & See \cref{j=0}   & \\ \hline
    $(\frac{3\alpha-5}{2}, 8)$  &  0 & See \cref{j=0}   & \\ \hline
    $(\frac{-3\alpha-5}{2}, \frac{3\alpha-11}{2})$  &  0  & See \cref{j=0}   & \\ \hline
    $(\frac{3\alpha-5}{2}, \frac{-3\alpha-11}{2})$  &  0  & See \cref{j=0}     & \\ \hline

    \end{tabular}
    \begin{tablenotes}
            \item[$\dagger$] The elliptic curve $y^2 = x^3+ax+b$ is denoted by $[a,b]$. In the last column we list the degrees of the irreducible factors of $f_C$ over $K$.
            \item[$\ddagger$] We use the model $y^2 + xy = x^3 - 4x - 1$ for $X_0(21)$.
        \end{tablenotes}
     \end{threeparttable}

\end{table}

}

On the other hand, over $K= \mathbb{Q}(i)$, $X_0(21)$ has rank 1. A search of points did not  produce an example of a 21-isogeny pointwise rational over a quadratic extension of $K$.

\subsection{}

{
\renewcommand{\arraystretch}{1.3}

\begin{table}[ht]

 \caption{($K=\mathbb{Q}(\sqrt{-1})$) Representatives (E,C) of isomorphism classes corresponding to non-cuspidal K-rational points  on a model of $X_0(15)$ } 
      \begin{threeparttable}

 \centering 

    \begin{tabular}{ | l | l | l | l |}

    \hline
\label{x0(15)}

    Point & j(E) &  E &  Deg($f_C$) \\ \hline

    $(8, -27)$& $-121945/32$ & $[-87/20, -421/100]$
 &  (1,1,1,2,2) \\ \hline
    
     (-2, -2)& $46969655/32768$ & $[633/54080, 239/1081600]$
 & (1,1,1,2,2) \\ \hline
    
    (-13/4, 9/8)& -25/2 & $[-12/25, -944/625 ]$ &  (1,2,4) \\ \hline
    
   (3,-2)  & -349938025/8 & $[\frac{-46272}{4225}, \frac{-1473536}{105625} ]$ &  (1,2,4) \\ \hline
   
    $(1/2,\frac{-15i - 3}{4})$ &  $\frac{-198261i-62613}{2}$  & 
    $[\frac{6846i + 9528}{105625}, \frac{-22652i + 30164}{2640625}] $&  (1,2,4)  \\ \hline
    
     $(1/2,\frac{15i - 3}{4})$ &  $\frac{198261i-62613}{2}$  & 
    $[\frac{-6846i + 9528}{105625}, \frac{22652i + 30164}{2640625}] $&  (1,2,4)  \\ \hline
    
     $(3i-1, -6i+6)$  & $\frac{15363i - 47709}{256}$  & $[\frac{-3i+96}{200}, \frac{3989i - 373}{10000}]$ &  (1,2,4) \\ \hline 
    $(-3i-1, 6i+6)$  & $\frac{-15363i - 47709}{256}$  & $[\frac{3i+96}{200}, \frac{-3989i - 373}{10000}]$ & (1,2,4) \\ \hline

    $(3i-1, 3i-6)$ &  $\frac{-13670181i+19928133}{8}$ & $ [ \frac{2583*i + 9444}{8450}, \frac{-93373i+39511}{211250} ] $  & (1,2,4) \\ \hline
    $(-3i-1, -3i-6)$ &  $\frac{13670181i+19928133}{8}$ & $ [ \frac{-2583*i + 9444}{8450}, \frac{93373i+39511}{211250} ] $  & (1,2,4) \\ \hline
    
     $(-7, 15i+3)$ &  $\frac{-86643i-1971}{4}$  &  $[\frac{216i-2688}{625}, \frac{8608i-53344}{15625} ]$ & (1,2,4) \\ \hline
     
      $(-7, -15i+3)$ &  $\frac{86643i-1971}{4}$  &  $[\frac{-216i-2688}{625}, \frac{-8608i-53344}{15625} ]$ &  (1,2,4) \\ \hline

\end{tabular}
  \begin{tablenotes}
            \item[$\dagger$]  In the last column we list the degrees of the irreducible factors of $f_C$ over K.
            \item[$\ddagger$] We use the model $y^2+xy+y=x^3+ x^2-10x-10$ for $X_0(15)$.
        \end{tablenotes}
     \end{threeparttable}

\end{table}
}

Now we will study 15-isogenies over $K = \mathbb{Q}(i)$.  
We see the first two entries in \Cref{x0(15)} indicate the only potential isomorphism classes in which we could find a 15-isogeny pointwise rational over a quadratic extension of $K$. 
Hence if a pair $(E,C)$ exists with $E/K$, $C$ Gal($\overline{K}/K$)-invariant and the points of $C$
$L$-rational for some quadratic extension $L/K$
then in fact $E$ is defined over $\mathbb{Q}$ and $C$ is Gal($\overline{\mathbb{Q}}/\mathbb{Q})$-invariant.

The point $(8,-27)$ corresponds to $(E,C)$ with 
$$ f_C =  (x - 7/10)(x+ 1/2)(x+ 17/10)(x^2+x- 139/20)(x^2+13x+ 269/20)$$
A brief computation yields $K(f_C)=\mathbb{Q}(\sqrt{5})$. Since $j \not = 0, 1728$, any pair $(E',C')$ equivalent to $(E,C)$ is of the form $E'=E^d$, $C'=C^d$ for some $d$ in $K$. As $K(f_C) = K(f_{C^d})$, if $K_C/K$ is degree 2 
then we must have $K_C = K(\sqrt{5})$.
The point $(1/2, 3\sqrt{-6}/5  )$ is in $C$, so the only potential $d$-twists (up to a square in $K$) in which $K_{C^d}/K$ is degree 2
(namely $K(\sqrt{5})$) are $d=-6,-6 \cdot 5$.  Magma now tells us that for these two values of $d$, $E^d(K(\sqrt{5}))_{tor} \approx C_{15}$. 

The point $(-2,-2)$ corresponds to $(E,C)$ with $f_C = l(x)q(x)$ where: 
$$l(x) =  (x - 3/104)(x+17/520)(x+113/520)$$
\vspace{-5mm}
$$q(x) = (x^2 - (11/52)x + 2333/54080)(x^2 + (1/52)x + 437/54080)$$
A brief computation yields $K(f_C)=\mathbb{Q}(\sqrt{-15})$, so as above, if $K_C/K$ is degree 2 then we must have $K_C =
K(\sqrt{-15})$.
The point $(3/104, 4\sqrt{26}/845  )$ is in $C$, so the only potential $d$-twists in which $K_{C^d}/K$ is degree 2
are $d=26,26 \cdot (-15)$.  Magma now tells us that for these two values of $d$, $E^d(K(\sqrt{-15}))_{tor} \approx C_{15}$.  Hence there are \textit{exactly} four elliptic curves over $K$ (up to isomorphism over $K$) with a 15-isogeny pointwise rational over a quadratic extension of $K$.

Similarly, when $K = \mathbb{Q}(\sqrt{-3})$ we find the same four elliptic curves are  the only elliptic curves over $K$ with a 15-isogeny pointwise rational over a quadratic extension of $K$.

\subsection{}

{
\renewcommand{\arraystretch}{1.3}

 \begin{table}[ht]
 \caption{($K=\mathbb{Q}(\sqrt{-1})$) Representatives (E,C) of isomorphism classes corresponding to non-cuspidal K-rational points  on a model of $X_0(20)$   } 
  \begin{threeparttable}

\setlength{\tabcolsep}{14pt}

    \begin{tabular}{| l | l | l | l |}
    
    \hline
\label{x0(20)}    
    Point & j(E) &  E & $f_C$ \\ \hline
    
    $(-2i, 0)$& $287496$ & $[\frac{264i + 77}{625}, \frac{616i + 1638}{15625}]$ & $(1,1,2,2,4)$ \\ \hline
    
    $(2i, 0)$  & $287496$ & $[\frac{-264i + 77}{625}, \frac{-616i + 1638}{15625}]$  & $(1,1,2,2,4)$ \\ \hline

   $(2i-2, -2i-4)$  & $287496$ & $[\frac{264i + 77}{625}, \frac{616i + 1638}{15625}]$ & $(1,1,2,2,4)$ \\ \hline

   $(-2i-2, 2i-4)$  & 287496 & $[\frac{-264i + 77}{625}, \frac{-616i + 1638}{15625}]$ & $(1,1,2,2,4)$ \\ \hline
   
   $(2i-2, 2i+4)$  & 1728 & See \cref{j=1728}   &   \\ \hline
   
   $(-2i-2, -2i+4)$  & 1728 & See \cref{j=1728} &    \\ \hline

    \end{tabular} 
     \begin{tablenotes}
            \item[$\dagger$]  In the last column we list the degrees of the irreducible factors of $f_C$ over K.
            \item[$\ddagger$] We use the model $y^2=x^3+x^2+4x +4$  for $X_0(20)$.
        \end{tablenotes}
     \end{threeparttable}

\end{table}
}

 Over $K=\mathbb{Q}(\sqrt{-3})$, $X_0(20)$ has rank 0, torsion $C_6$ and these points are all cusps.  Over $K=\mathbb{Q}(\sqrt{-1})$, $X_0(20)$ has rank 0, torsion $C_2 \oplus C_6$ and 6 cusps.  \Cref{x0(20)} shows that 4 non-cuspidal K-points correspond to 20-isogenies pointwise rational over extensions of $K$ of degree at least 4. If $E/K$ has $j(E)=1728$ and a 20-isogeny pointwise rational over a quadratic extension of $K$ then some quadratic twist of $E$  has a point of order 5 over $K$ by \cref{oddpartsadd}.  But this contradicticts \cref{j=1728}.
Therefore there are no 20-isogenies pointwise rational over quadratic extensions of $K=\mathbb{Q}(\sqrt{-3})$ or $\mathbb{Q}(i)$.

\subsection{}

Over $\mathbb{Q}$, $X_0(24)$ has rank 0, torsion $C_2 \oplus C_4$ and 8 cusps.  The torsion and rank do not grow upon extension to $K = \mathbb{Q}(\sqrt{-3})$ or $\mathbb{Q}(i)$ so there are no elliptic curves over $K$ with a 24-isogeny (pointwise rational over \textit{any} extension of $K$).

\subsection{}

The curve $X_0(27)$ is an elliptic curve with model $y^2 + y = x^3 - 7$.
Over $\mathbb{Q}$, $X_0(27)$ has rank 0, torsion $C_3$ and 2 cusps.  The torsion and rank do not grow upon extension to $K =  \mathbb{Q}(i)$  The one non-cuspidal point (3,-5) corresponds to a pair $(E,C)$ with $j(E) = -12288000$ and the degrees of the irreducible factors of $f_c$ over $\mathbb{Q}(i)$ are (1,3,9).  Because $j(E) \not  = 0,1728$, the isomorphism class of $(E,C)$ just consists of quadratic twists of this pair, and hence will yield the same degrees of irreducible factors. Over $K = \mathbb{Q}(\sqrt{-3})$, $X_0(27)$ has 6 cusps and $X_0(27)(K)=C_3 \oplus C_3$.  As in the case $K=\mathbb{Q}(i)$, the 3 non-cuspidal points do not yield 27-isogenies pointwise rational over a quadratic extension of $K$.  Therefore in either case, there are no elliptic curves over $K$ with a 27-isogeny pointwise rational over a quadratic extension of $K$.

\subsection{}

Let $K=\mathbb{Q}(\sqrt{-3})$ or $\mathbb{Q}(i)$.  If $E/K$ possesses a cyclic Gal($\overline{K}/K$)-invariant subgroup $C$ of order 30, then $C$ has a unique cyclic subgroup of order 15 and hence this subgroup is Gal($\overline{K}/K$)-invariant as well.  So if $K_C/K$ is degree 1 or 2 
then $E$ possesses a 15-isogeny pointwise rational over a quadratic extension of $K$.  
But there are only four such pairs $(E,C')$, and we found that in each case $K_{C'}/K$ was degree 2
so we would have $K_C = K_{C'}$.
But as already noted, the torsion over the extension was $C_{15}$ in each case, so there are no 30-isogenies over $K$ (pointwise rational over a quadratic extension of $K$).

\subsection{}

Magma tells us $X_0(35)$ is genus 3 with affine model $$ y^2 + (-x^4 - x^2 - 1)y = -x^7 - 2x^6 - x^5 - 3x^4 + x^3 - 2x^2 + x $$

Furthermore Magma found an automorphism of $X_0(35)$ such that the quotient curve $E$ is genus 1 with affine model:
$$y^2 + y = x^3 + x^2 + 9x + 1$$
The quotient map (defined between the projective closures) is given by:

$$f:X_0(35) \rightarrow E $$
$$(x,y,z) \mapsto (p^f_1,p^f_2,p^f_3)$$
 
 \noindent $p^f_1 = x^4 - 5x^3z - 8x^2z^2 + 5xz^3 + z^4$\\
$p^f_2 = 3x^4 - x^3z + 4x^2z^2 + xz^3 - 7yz^3 + 3z^4$\\
$p^f_3 = x^4 + 2x^3z - x^2z^2 - 2xz^3 + z^4$\\

Because $f$ is a rational map, the only potential $K$-rational points of $X_0(35)$ are the non-regular points of $f$ and $f^{-1}(E(K))$. In order to compute $f^{-1}(E(K))$ we must first compute $E(K)$.  Magma/Sage give us the following information:

\begin{table}[ht]
\caption{ $K$-Rational Points on a Genus 1 Quotient of $X_0(35)$}
\begin{center}
    \begin{tabular}{| l | l | l | p{5cm} |}
    \hline
    K &  rk($E(K)$) & $E(K)_{tor}$ & Points of $E(K)_{tor}$\\ \hline
    $\mathbb{Q}(\sqrt{-1})$ & 0 & $ C_3$ & [0,1,0], [1,3,1],[1,-4,1] \\ \hline
     $\mathbb{Q}(\sqrt{-3})$ & 0 & $C_3 \oplus C_3$ & 
     $[0,1,0], [1,3,1],[1,-4,1], \newline
     [\frac{1}{2}(5\alpha-1), \frac{1}{2}(-5\alpha+9),1], \newline
     [\frac{1}{2}(-5\alpha-1), \frac{1}{2}(5\alpha+9),1],\newline
     [\frac{1}{2}(5\alpha - 1), \frac{1}{2}(5\alpha-11),1],\newline
     [\frac{1}{2}(-5\alpha-1), \frac{1}{2}(-5\alpha-11),1],\newline
     [-\frac{4}{3}, \frac{1}{18}(35\alpha -9),1 ],\newline
     [-\frac{4}{3}, \frac{1}{18}(-35\alpha -9),1 ]$

    \\ \hline
    \end{tabular}
\end{center}
\end{table}

To compute $f^{-1}([x,y,z])$, we form the ideal $<C, p^f_1-wx, p^f_2-wy,p^f_3-wz>$ 
(C denotes the model of $X_0(35)$ above) and compute its Gr{\"o}bner Basis (with respect to the ordering x,y,z,w).  Often, one can find basis elements that allow the system to be solved by hand.  We can assume $z\not = 0$ as the only point on our model of $X_0(35)$ with this property is [0,1,0] and $f$ is not defined at this point.

\begin{table}[ht]
\caption{ Gr{\"o}bner basis data for determination of $f^{-1}(E(K))$ }
\begin{center}
    \begin{tabular}{| l | l | l |  p{10cm} |}
    \hline
    Point P of E(K) & $f^{-1}(P)$ & Gr{\"o}bner basis elements \\ \hline
    [0,1,0] & $\emptyset$  & $w^2$  \\ \hline
    
    [1,3,1] & [0,0,1]  & $xw^2,yw^2$ \\ \hline
    [1,-4,1] &    [0,1,1]     &  $xw^2, yw^2-zw^2$ \\ \hline

    \end{tabular}
\end{center}

\end{table}

For each of the six extra points over $K = \mathbb{Q}(\sqrt{-3})$ the Gr{\"o}bner basis contains a polynomial $g(w)$. Using Magma one can check that in each case the only root of $g(w)$ over $K$ is 0.  Hence the ($K$-rational) inverse image of these points under $f$ is empty.

Finally, using a Gr{\"o}bner basis for the ideal $<C, p^f_1, p^f_2,p^f_3>$ we can determine the non-regular points of $f$.  If $z \not = 0$ then the Groebner basis contains $y^2-6y+4$.  This has no roots over $K$.  Hence the only $K$-rational points on $X_0(35)$ are $[0,0,1]$, $[0,1,0]$ and $[0,1,1]$.  These points are all cusps 
so there are no 35-isogenies defined over $K = \mathbb{Q}(\sqrt{-1}),\mathbb{Q}(\sqrt{-3})$.

We summarize our findings in the following theorem.

\begin{thm}
\label{isogenies}

If K $=\mathbb{Q}(\sqrt{-3})$ and $E/K$ is an elliptic curve, $E$ has no $N$-isogenies pointwise rational over a quadratic extension of $K$ for $N= 20,21,24,27,30,35,45$.  If $K =\mathbb{Q}(\sqrt{-1})$ and $E/K$ is an elliptic curve, $E$ has no $N$-isogenies pointwise rational over a quadratic extension of $K$ for $N= 20,24,27,30,35,45$.  The curve $X_0(21)$ is genus 1 and rank 1 over $K$.  In either case above, there are exactly four elliptic curves over $K$ (up to isomorphism over $K$) with a 15-isogeny pointwise rational over a quadratic extension of $K$.

\end{thm}

\section{Growth of Torsion}

\begin{thm}
\label{oddtwist}
Let K $= \mathbb{Q}(\sqrt{-3}$),  d $\in$ $\mathcal{O}_K$, d a nonsquare, and E/K an elliptic curve.

\begin{enumerate}

\item If $E(K)_{tor} \approx C_7, C_9 \mbox{ or }  C_3 \oplus C_3, $ then  $E^d(K)_{tor} \approx C_1 $

\item If $E(K)_{tor} \approx C_3 $ then  $E^d(K)_{tor} \approx C_1 \mbox{ or } C_5 $

\item If $E(K)_{tor} \approx C_5 $ then  $E^d(K)_{tor} \approx C_1 \mbox{ or } C_3 $

\item If $E(K)_{tor} \approx C_1 $ then  $E^d(K)_{tor} \approx C_1, C_3, C_5, C_7, C_9 \mbox{ or } C_3 \oplus C_3 $

\end{enumerate}
Hence the torsion structures $C_7, C_9$ and $C_3 \oplus C_3$ do not grow in any quadratic extension of K.  The torsion structures $C_3$ and $C_5$ grow in at most 1 extension, and $C_1$ grows in at most 2 extensions.

\end{thm}

\begin{proof}
Let $d \in K$ be a non-square.  Note that if $E'$ is a quadratic twist of $E$ then $E$ is a quadratic twist of $E'$ (up to isomorphism over $K$).  Also by \cref{GT}, all quadratic twists of a curve with odd order torsion will be odd order. By 
\cref{Najman},
the only odd-order torsion structures occurring over $K$ are $C_1,C_3,C_5,C_7,C_9$ and $C_3 \oplus C_3$.
Now if $E(K)[3] \not = C_1$ and $E^d(K)[3] \not = C_1$ then by \cref{oddpartsadd}, $E(K(\sqrt{d}))[3] = C_3 \oplus C_3$, $C_3 \oplus C_3 \oplus C_3$ or $C_3 \oplus C_3 \oplus C_3 \oplus C_3$ , contradicting \cref{GT} Part 5 or \cite{Si}*{Cor. 6.4} respectively. If $m = 5$ or $7$, $E(K)_{tor} \approx C_m$ and $E^d(K)_{tor} \approx C_m$, then by \cref{twistback}, $C_m \oplus C_m \subseteq$ E(L), contradicting \cref{full}.  If $E(K)[3] \not = C_1$ and $E^d(K) \approx C_7$, then by \cref{twistback} $E^d$ has a 21-isogeny pointwise rational over a quadratic extension of $K$.  But no such isogeny exists by \cref{isogenies}.
If $E(K)[5]  = C_5$ and $E^d(K) \approx C_7$, then by \cref{twistback} $E^d$ has a 35-isogeny pointwise rational over a quadratic extension.  But no such isogeny exists by \cref{isogenies}. If $E(K)[3] \not = C_1$  and $E^d(K) \approx C_5$, then by \cref{twistback} $E^d$ has a 15-isogeny pointwise rational over a quadratic extension of $K$.  There are four elliptic curves (two pairs of quadratic twists) over $K$ (up to isomorphism over $K$) with 
such an isogeny.
For each such curve $E$ (we actually need only check one member of each pair), the factorization of the 3-division polynomial of $E$ indicates that the nontrivial torsion structures occurring among the quadratic twists of E are $C_3$ and $C_5$ and each occurs exactly once.
\end{proof}

\textbf{Acknowledgements}

We are grateful to Filip Najman for his   generous advice and for pointing out several issues with an earlier draft of this paper, as well as thankful to Sheldon Kamienny for his helpful discussions.

\end{document}